\documentclass[a4paper,10pt]{amsart}

\usepackage{amssymb,amsthm,amsmath,verbatim}
\usepackage{t1enc}

\usepackage{enumerate}
\usepackage{multicol}

\usepackage{color}
\usepackage{tikz}
\usepackage{breqn}
\usepackage{xmpmulti}
\usepackage{psfrag}

 \usepackage[usenames,dvipsnames]{pstricks}
 \usepackage{epsfig}
\usepackage{pst-grad}
\usepackage{pstricks,pst-node}
\usepackage{float}

\usepackage{wrapfig}
\usepackage{framed}

\usepackage[left=4 cm, right=4cm]{geometry}
\numberwithin{equation}{section}

\newenvironment{tproof}{
  
  \begin{proof}
}{\end{proof}}

\newenvironment{cproof}{
  
  \begin{proof}
}{\end{proof}}

\newtheorem{prop}{Proposition}[section]

\newtheorem{lemma}[prop]{Lemma}

\newtheorem{cor}[prop]{Corollary}

\newtheorem{thm}[prop]{Theorem}

\newtheorem{clm}[prop]{Claim}

\newtheorem{obs}[prop]{Observation}
\newtheorem{subclaim}{Subclaim}[prop]

\newtheorem{quest}[prop]{Question}

\newcommand{\mc}[1]{\mathcal{#1}}
\newcommand{\mb}[1]{\mathbb{#1}}

\newcommand{\oo}{\omega}
\newcommand{\uhr}{\upharpoonright}
\newcommand{\omg}{{\omega_1}}

\usepackage{todonotes}
\newcommand{\piii}{\Pi^1_1}
\newcommand{\siii}{\Sigma^1_1}
\newcommand{\diii}{\Delta^1_1}

\DeclareMathOperator{\atree}{AT}
\DeclareMathOperator{\stree}{ST}
\DeclareMathOperator{\otree}{TO}
\DeclareMathOperator{\ktree}{KT}
\DeclareMathOperator{\satree}{sAT}
\DeclareMathOperator{\Borel}{Borel}

  \DeclareMathOperator{\stat}{Stat}

\DeclareMathOperator{\htt}{ht}

\def\<{\left\langle}
\def\>{\right\rangle}
\def\br#1;#2;{\bigl[ {#1} \bigr]^ {#2} }

\newcommand{\setm}{\setminus}

\newcommand{\subs}{\subset}

\newcommand{\ran}{\operatorname{ran}}

\newcommand{\smf}{\hspace{0.008 cm}^\smallfrown}

\title[]{On the complexity of classes of uncountable structures: trees on $\aleph_1$}

\date{\today}

  \author{Sy-David Friedman}
  \address[Sy-D. Friedman]{Universit\"at Wien,
Kurt G\"odel Research Center for Mathematical Logic, Wien, Austria}
 \email{sdf@logic.univie.ac.at}
  
  \author{D\'aniel T. Soukup}
  \address[D.T. Soukup]{Universit\"at Wien,
Kurt G\"odel Research Center for Mathematical Logic, Wien, Austria}
 \email[Corresponding author]{daniel.soukup@univie.ac.at}

 \urladdr{http://www.logic.univie.ac.at/$\sim  $soukupd73/}

\makeatletter
\newtheorem*{rep@theorem}{\rep@title}
\newcommand{\newreptheorem}[2]{%
\newenvironment{rep#1}[1]{%
 \def\rep@title{#2 \ref{##1}}%
 \begin{rep@theorem}}%
 {\end{rep@theorem}}}
\makeatother

\newreptheorem{corollary}{Corollary}
\newreptheorem{theorem}{Theorem}

\subjclass[2010]{}

\begin{document}
 \begin{abstract}  
We analyse the complexity of the class of (special) Aronszajn, Suslin and Kurepa trees in the projective hierarchy of the higher Baire-space $\omg^\omg$. First, we will show that none of these classes have the Baire property (unless they are empty). Moreover, under $(V=L)$, (a) the class of Aronszajn and Suslin trees is $\piii$-complete, (b) the class of special Aronszajn trees is $\siii$-complete, and (c) the class of Kurepa trees is $\Pi^1_2$-complete.  We achieve these results by finding nicely definable reductions that map subsets $X$ of $\omg$ to trees $T_X$ so that $T_X$ is in a given tree-class $\mc T$ if and only if $X$ is stationary/non-stationary (depending on the class $\mc T$). Finally, we present models of CH where these classes have lower projective complexity.
  
 \end{abstract}

 \maketitle

 \section{Introduction}
 
We set out to investigate the complexity of certain well-studied classes of $\aleph_1$-trees on $\omega_1$. In particular, under various set-theoretic assumptions, we determine the Borel/projective complexity of the class of  Aronszajn, Suslin and Kurepa trees as a subset of the higher Baire space $\omg^\omg$. In several cases, we prove the existence of nicely definable reductions between  these tree classes and the stationary relation on $\mc P(\omega_1)$. As the latter is $\Pi^1_1$-complete under $V=L$, we get the parallel completeness of the tree classes. In the case of Kurepa-trees, we use a different coding argument.

 \medskip
 
The general setting of our paper is \textit{the higher Baire space}\footnote{The name generalized Baire space is also commonly used.} on $\omg^\omg$ and $2^\omg$. Basic open sets correspond to countable partial functions which, in turn, give rise to a $\omg$-Borel structure on $\omg^\omg$, $2^\omg$ and so $\mc P(\omg)$ as well. This allows us to measure the complexity of subsets of $\omg^\omg$ or, equivalently, of families of natural combinatorial structures on $\omg$. In this paper, we will focus on models of CH i.e., $2^{\aleph_0}=\aleph_1$. This is a fairly natural assumption in this higher Baire setting which, in particular, ensures that $\omg^\omg$ has a basis of size $\aleph_1$. This, of course, is analogous to the standard Baire space $\oo^\oo$ having a countable basis.
 
 Our primary interest lies in the set of \emph{$\aleph_1$-trees}: partial orders $T=(\omg,<_T)$ on $\omega_1$ so that (1) the set of predecessors of each node is well-ordered,\footnote{This allows us to define a height function on $T$ and the levels $T_\xi$ of $T$} and (2) for any ordinal $\xi$, the set of nodes $T_\xi$ of height $\xi$ is countable and non-empty if $\xi<\omg$ and empty otherwise.
 
 
 Following \cite{vaananentrees}, we will use $\otree$ to denote the class of trees without uncountable branches; so we allow trees with uncountable levels here. A tree $T$ in $\otree$ is called \emph{Aronszajn} if $T$ is also an $\aleph_1$-tree (i.e., all levels are countable). We call $T$ a  \emph{Suslin tree} if it is an Aronszajn tree without uncountable antichains. On the other hand, an $\aleph_1$-tree $T$ is \emph{Kurepa} if it has at least $\aleph_2$ uncountable branches. We will denote the classes of these trees with $\atree$, $\stree$ and $\ktree$, respectively. An Aronszajn tree is \emph{special} if it is the union of countably many antichains.  The latter collection will be denoted by $\satree$. These are the main classes of trees we will be analyzing in detail. Let us refer the reader to the classical set theory textbooks \cite{kunen, jech2003set}  and to \cite{stevotrees} for a nice introduction to trees of height $\aleph_1$; the latter survey emphasizes the connection of trees to topology and linear orders. 
  
 \medskip
 
 \begin{figure}[H]
     \centering
     
\psscalebox{1.0 1.0} 
{
\begin{pspicture}(0,-2.739834)(12.1,2.739834)
\psframe[linecolor=black, linewidth=0.04, dimen=outer](7.6,1.260166)(0.0,-1.939834)
\rput[bl](0.0,1.460166){\Large{$\atree$: Aronszajn trees}}
\psframe[linecolor=black, linewidth=0.04, linestyle=dashed, dash=0.17638889cm 0.10583334cm, dimen=outer](11.2,1.260166)(8.8,-1.939834)
\rput[bl](9.2,0.260166){\Large{$\ktree$:}}
\rput[bl](9.5,-0.43983397){\Large{Kurepa}}
\rput[bl](9.8,-0.93983397){\Large{trees}}
\psellipse[linecolor=black, linewidth=0.04, dimen=outer](2.0,-0.33983397)(1.6,1.2)
\psellipse[linecolor=black, linewidth=0.04, linestyle=dashed, dash=0.17638889cm 0.10583334cm, dimen=outer](5.6,-0.33983397)(1.6,1.2)
\rput[bl](0.8,-0.23074308){\Large{$\satree$: special}}
\rput[bl](1.2,-0.83074308){\Large{Aronszajn}}

\rput[bl](4.628571,-0.13983397){\Large{$\stree$: Suslin}}
\rput[bl](5.48571,-0.73983397){\Large{trees}}
\psframe[linecolor=black, linewidth=0.04, dimen=outer](8.4,2.2)(-0.5,-2.739834)
\rput[bl](0,2.460166){\Large{$\otree$: trees with no uncountable branch}}
\end{pspicture}
}

     \caption{Classes of trees of height $\omg$}
     \label{fig:trees}
 \end{figure}
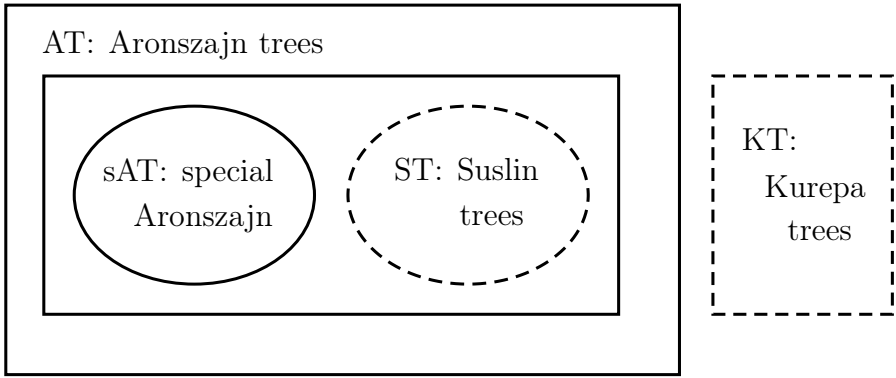
 
 \medskip
 
 Recall that special Aronszajn trees exist in ZFC, however, even assuming the Continuum Hypothesis, $\stree$ and $\ktree$ may be empty. In fact, as proved by R. Jensen, CH is consistent with $\atree=\satree$ \cite{devlin2006souslin} and so $\stree = \emptyset$ in this model.\footnote{We mention that $\atree\neq \satree$ does not imply the existence of Suslin-trees \cite{shelah2017proper}.} The consistency of no Kurepa trees was prove by Silver \cite{silver1971independence}. 
 
 On the other hand, under $V=L$ (or just assuming strong enough diamonds), both Suslin and Kurepa-trees exist \cite{jech2003set}. 
 
 \medskip
 
Let us present some results that will place our paper in the context of past research. Trees have played a significant role in the study of both the standard and higher Baire space \cite{kechris1987descriptive, vaananentrees, vaananen1995games, mirna, mekler1993canary}. Recall that the set of trees on $\oo$ without infinite branches is complete co-analytic. Analogously, a classical result from the theory of higher descriptive set theory is the following theorem of J. V\"a\"an\"anen.

  
  \begin{thm} \cite{vaananentrees} CH implies that $\otree$, the set of all trees on $\omg$ without uncountable branches, is $\piii$-complete.
  \end{thm}
  
    We mention that $\otree$ is $\siii$ if $MA_{\aleph_1}$ holds; indeed,  $MA_{\aleph_1}$ implies that $\otree$ is exactly the set of special trees on $\omg$ \cite{baumgartner1970embedding} which is easily verified as a $\siii$ definition.

  Yet another subclass of $\otree$ is the following: a \emph{canary tree} $T$ is a tree of size continuum with no uncountable branches with the property that in any extension $W$ of the universe $V$ with $\mb R^V=\mb R^W$, if a stationary set of $V$ is no longer stationary in $W$ then $T$ has an uncountable branch in $W$. Now, canary trees give a simple definition for a subset of $\omg$ to be stationary.

   \begin{thm} \cite{mekler1993canary,hyttinen2001canary} There is a Canary-tree iff $\stat\subset \omg^\omg$, the set of all stationary subsets of $\omg$, is $\siii$. Moreover, the existence of Canary trees is independent of GCH.
  \end{thm}
 
If $V=L$ then there are no canary trees and in fact, the following polar opposite result holds which appears implicitly in \cite{fokina2013classes}.

\begin{thm} \cite{fokina2013classes}\label{thm:fok} If $V=L$ then $\stat$ is $\piii$-complete.
\end{thm}

 We will use this theorem to show that certain classes of trees are complete in their complexity class. Finally, let us mention that there are strong connections between infinitary logic, trees and the complexity questions that our paper is concerned with \cite{vaananen1995games, shelah2000stationary, fokina2013classes}. We only  included the results most relevant for our studies but we would like to refer the reader to the survey \cite{vaananen1995games} and the book \cite{friedman2014generalized} for more details.
 
 \medskip
 
First, we will start by showing that non of the classes $\atree,\satree,\stree$ and $\ktree$ have the Baire property and hence they are non Borel (unless $\stree$ and $\ktree$ are empty, in which case they are trivially Borel). Moreover, we will prove the following results about the complexity of these classes:
 \medskip

\begin{figure}[H]
    \centering
    \bgroup
\def\arraystretch{1.5}
  \begin{tabular}{|c|c|c|c|c|}
  \hline
     & $\atree$ & $\satree$ & $\stree$ & $\ktree$ \\
\hline
ZFC & $\piii\setm \Borel$ & $\siii\setm \Borel$ & $\piii\setm \Borel$ or $\emptyset$ & $\Pi^1_2\setm \Borel$ or $\emptyset$ \\
     \hline
     $V=L$ & $\piii$-complete &  $\siii$-complete & $\piii$-complete  & $\Pi^1_2$-complete \\
     \hline
     Abraham-Shelah model &$\diii\setm \Borel$ & $\diii\setm \Borel$ & $\diii\setm \Borel$ &
     ?\\
     \hline
     $MA_{\aleph_1}$ or $PFA(S)[S]$ & \multicolumn{2}{|c|}{$\diii\setm \Borel$} & $\emptyset$  &
     ?\\
     \hline
\end{tabular}
\egroup
    \caption{A summary of complexity}
    \label{fig:table}
\end{figure}

Some of these results are easy consequences of known theorems (such as the results regarding the Abraham-Shelah model which we will describe shortly). However, the completeness of the classes under $V=L$ requires significant work and new ideas. We present the ZFC results and facts about the Abraham-Shelah model in Section \ref{sec:ZFC}. Then, in Section \ref{sec:reduce}, we will show that stationarity can be reduced (in a Borel way) to the classes $\atree$ and $\stree$. The results on $V=L$ will follow easily then. Finally, we deal with Kurepa-trees in Section  \ref{sec:kurepa}. Note that under $MA_{\aleph_1}$ and $PFA(S)[S]$, $\atree=\satree$ so the last line of Figure \ref{fig:table} follows by wrapping out the definitions.\footnote{In fact, in the latter model, all Aronszajn trees are club-isomorphic \cite{yorioka2017club}.} We end our paper with some remarks and open problems in Section \ref{sec:questions}.


 \subsection{Preliminaries} 
 
 
 
 We defined the tree classes already but let us review the most important descriptive set theoretic notions that we need. The family of \textit{Borel sets} in $\omg^\omg$ is the smallest family containing all open sets which is closed under taking complements and unions/intersections of size $\aleph_1$. It is easy to see that the set of $\aleph_1$-trees forms a Borel set  with an appropriate coding of the order into a subset of $\omg$. 
 
 Now, a subset $\mc T$ of  $\omg^\omg$ is $\piii$ (and called \textit{co-analytic}) if there is an open $B\subset \omg^\omg \times \omg^\omg$ so that $T\in \mc T$ if and only for all $g\in \omg^\omg$, $(T,g)\in B$. Complements of $\piii$ sets, denoted by $\siii$, are called \textit{analytic} sets. In Section \ref{sec:ZFC}, the reader can see elementary applications of this definition.
 
 Finally, a subset $\mc T$ of  $\omg^\omg$ is \textit{complete} for a complexity class $\Gamma$ iff $\mc T\in \Gamma$ and for any $\mc S\in \Gamma$, there is a continuous $\pi:\omg^\omg \to \omg^\omg$ so that $T\in \mc S$ if and only if $\pi(T)\in \mc T$. That is, no matter how we pick $\mc S$ in $\Gamma$, we can completely decide $\mc S$ by our single fixed set $\mc T$ and using an appropriate continuous map. 
 In Section \ref{sec:reduce}, we shall see this definition at work.
 
 \medskip
 
 Let us also recall some classical guessing principles:  $\diamondsuit^+$ asserts the existence of a sequence $\underline A=\{\mc A_\alpha:\alpha<\omg\}$ of countable sets so that for any $X\subset \omg$, there is a club $C\subs \omg$ so that $C\cap \alpha,X\cap \alpha\in \mc A_\alpha$ for any $\alpha\in C$. In this situation, we say that $\underline A$ witnesses $\diamondsuit^+$.
 
We say that $\underline N=(N_\alpha)_{\alpha<\omg}$ is a \emph{$\diamondsuit^+$-oracle over $P$} if  
\begin{enumerate}
    \item $\underline N$ is an increasing sequence of countable elementary submodels of $H(\aleph_2)$,
    \item $P,(N_\alpha)_{\alpha<\beta}\in N_\beta$ for all $\beta<\omg$, and
    \item $\underline N$ witnesses $\diamondsuit^+$.
\end{enumerate}

Clearly, if $\diamondsuit^+$ holds then for any $P\in H(\aleph_2)$, there is a $\diamondsuit^+$-oracle over $P$. Also, recall that $\diamondsuit^+$ implies that $\stree$ and $\ktree$ are non-empty \cite{jech2003set}.

\medskip

For later reference, we state a few consistency results, the first being a now classical theorem of R. Jensen.

\begin{thm}\cite{devlin2006souslin}\label{thm:jensen} Consistently, CH holds and all Aronszajn-trees are special.
\end{thm}

Jensen's argument was built on an elaborate ccc forcing (in fact, a completely new iteration technique).  A more mainstream proof of this theorem is due to S. Shelah \cite{shelah2017proper} using countable support iteration of proper posets.

\medskip

Given two trees $S,T$, a \emph{club-embedding of $T$ into $S$} is an order preserving injection $f$ defined on $T\uhr C= \bigcup\{T_\alpha:\alpha\in C\}$, where $C\subs \omg$ is a club (closed and unbounded subset), with range in $S$. A \textit{derived tree of $S$} is a level product of the form $\Pi_{i<n}S\cap s_i^{\uparrow}$ where the  $s_i$ are distinct nodes from the same level of $S$.\footnote{Here, $S\cap s_i^{\uparrow}=\{t\in S: t\geq s_i\}$.} A \textit{fully Suslin tree} is a Suslin tree with the property that all its derived trees are Suslin as well.

We will refer to the model in the next theorem as the \emph{Abraham-Shelah model}.

\begin{thm}\cite{abrahamiso} Consistently, CH holds and there is a fully Suslin tree $R$ and special Aronszajn tree $U$ so that, for any Aronszajn tree $T$, either
\begin{enumerate}
    \item $T$ club-embeds into $U$ or
    \item  there is a derived tree of $R$ that club-embeds into $T$.
\end{enumerate} Moreover, there are only $\aleph_1$-many Suslin-trees modulo club-isomorphism.\footnote{It is an intriguing open problem if one can find \textit{a model with a single Suslin-tree} (modulo club-isomorphism).}
\end{thm}

In essence, the above theorem says that any Aronszajn tree is either special or embeds a Suslin tree closely associated to $R$. 

\medskip

Finally, the fact that there might be no Kurepa trees was proved by J. Silver in 1971.

\begin{thm}\cite{silver1971independence} If a strongly inaccessible cardinal is L\'evy collapsed to $\omega_2$ then in the resulting model, there are no Kurepa trees. 
\end{thm}

 
 \subsection{Acknowledgments}
The authors would like to thank the Austrian Science Fund (FWF) for the generous support through Grant I1921. The second author was also supported by  NKFIH OTKA-113047.

 \section{Aronszajn and Suslin trees}\label{sec:ZFC}

  To avoid some technicalities, let us restrict our attention to certain \emph{regular trees} only from now on: those trees $T$ which are rooted, every node in $T$ has at least two immediate successors and $T$ is pruned i.e., for any $s\in T$ of height $\alpha$ and any $\beta<\omg$ above $\alpha$, there is some $t\in T$ of height $\beta$ that extends $s$. These are simple Borel conditions and we assume that our classes $\atree$, $\stree$ and later $\ktree$ consist of only regular trees.
  
  

  
  
  \medskip
  
  Now, let us start the complexity analysis of these classes. Our first observation follows from the definitions immediately.
  
 \begin{obs}\label{obs:basic}\begin{enumerate}
     \item The set of all $\aleph_1$-trees on $\omg$ is Borel.
     \item  $\atree$ and $\stree$ are both $\piii$ sets.
     \item $\satree$ is $\siii$.
 \end{enumerate}
 \end{obs}
 \begin{cproof}
 The proof is a fairly standard exercise in descriptive set theory. To demonstrate the definitions, we prove that $\atree$ is $\piii$ and leave the rest to the interested reader. We need to find an open $B\subset \omg^\omg \times \omg^\omg$ so that $T\in \atree$ if and only for all $g\in \omg^\omg$, $(T,g)\in B$. Indeed, let $B$ denote the set of pairs $(T,g)$ so that $T$ is an  $\aleph_1$-tree on $\omg$ and $g\in \omg^\omg$ does not code an uncountable branch in $T$. Now, $B$ is open in the product of codes for $\aleph_1$-trees (a Borel set) and $\omg^\omg$. Indeed, if $g$ does not code an uncountable branch then either $g$ codes a countable branch in $T$ (i.e., there is a level of $T$ without any element of $g$) or $g$ codes two incomparable elements. Both cases can be witnessed by fixing a countable initial segment of $T$ and $g$ and hence, $B$ is open. Now, an $\aleph_1$-tree $T$ is Aronszajn if and only if for any $g$, $(T,g)\in B$.
 \end{cproof}
 
 So in models where $\atree=\satree$ e.g., in Jensen's model of CH from Theorem \ref{thm:jensen}, we get the following.
 
 \begin{cor}
Consistently, CH holds and $\atree=\satree\in \diii$ and so $\stree=\emptyset$.\footnote{We remark here that $MA_{\aleph_1}$ implies $\atree=\satree$ so they are both $\diii$ however CH fails.}
 \end{cor}

 Our next goal is to show that none of the classes $\atree,\satree$ and $\stree$ are Borel, that is, unless $\stree=\emptyset$ in which case it is trivially Borel. We will apply the following well-known fact.
 
 \begin{lemma}\label{lm:ext}
 Suppose that $T,T'$ are countable, rooted, binary branching, and pruned trees of height $\beta<\omega_1$. Then $T$ and $T'$ are isomorphic. In fact, any isomorphism $f:T_{\leq \alpha}\to T'_{\leq \alpha}$ with $\alpha<\beta$ extends to an isomorphism $\bar f:T\to T'$.
 \end{lemma}

The proof is an easy back-and-forth argument that we omit.
 This allows us to prove a new, relatively straightforward result.
 
\begin{lemma}\label{lm:at1}
Suppose that $T$ is a regular $\aleph_1$-tree and $\mc U$ is somewhere co-meager in the set of all regular trees on $\omega_1$. Then 

\begin{enumerate}
    \item there is an isomorphic copy $S$ of $T$ in $\mc U$, and 
    
\item $\mc U$ contains a tree with an uncountable branch. 
\end{enumerate}
\end{lemma}
\begin{proof} (1) Suppose that $\mc U=[S^0]\setm \bigcup\{Y_\xi:\xi<\omg\}$ where each $Y_\xi$ is a nowhere dense set of trees. I.e., any countable tree $S$ has a countable end-extension $S'$ so that any extension of $S'$ into a tree on $\omg$ is not in $Y_\alpha$. Here, $[S^0]$ denotes the (basic open) set of all trees extending $S^0$.

Now, we construct an increasing sequence of countable trees $(S^\xi)_{\xi<\omg}$ and isomorphisms $f_\xi:S^\xi\to T_{<\delta_\xi}$. Given  $(S^\xi)_{\xi<\zeta}$, we look at $S^{<\zeta}=\bigcup_{\xi<\zeta} S^\xi$. The latter is isomorphic to $T_{<\delta}$ where $\delta=\sup_{\xi<\zeta}\delta_\xi$ witnessed by $$f_{<\zeta}=\bigcup_{\xi<\zeta} f_\xi:S^{<\zeta}\to T_{<\delta}.$$ Define an end extension $\overline{S^{<\zeta}}$ of $S^{<\zeta}$ of height $\delta+1$ by adding upper bounds to exactly those branches $b\subset S^{<\zeta}$ so that the chain $f_{<\zeta}[b]$ has an upper bound in $T_\delta$. Clearly, there is an isomorphism $\overline{ f_{<\zeta}}:\overline{S^{<\zeta}}\to T_{\delta+1}$ that extends  $f_{<\zeta}$. Now, let $S^\zeta$ be an end-extension of $\overline{S^{<\zeta}}$ which cannot be extended to a tree on $\omega_1$ that is in $Y_\zeta$. This can be done since $Y_\zeta$ is nowhere dense. Finally, apply Lemma \ref{lm:ext} to extend $\overline{ f_{<\zeta}}$ to some isomorphism $f_\zeta:S^\zeta\to T_{<\delta_\zeta}$.

This finishes the construction and the tree $S=\bigcup\{S^\zeta:\zeta<\omg\}$ is as desired.

(2) Since there is a regular $\aleph_1$-tree $T$ which contains an uncountable branch, we can apply (1).
\end{proof}

We shall use the fact that any non-meager set with the Baire property is somewhere co-meager. The previous lemma and latter fact immediately yields the following corollaries.
 
 \begin{cor}\begin{enumerate}
 \item\label{it:nonm} The isomorphism class of any regular $\aleph_1$-tree $T$ is everywhere non-meager.
 \item \label{it:sep} Suppose that $T,S$ are non-isomorphic $\aleph_1$-trees. Then their isomorphism classes cannot be separated by sets with the Baire-property.
 \item\label{it:x} The set of trees isomorphic to a fixed tree without an uncountable branch is $\siii$ but does not have the Baire property and hence is not Borel.
     \item\label{it:y} The sets $\atree$ and $\satree$ do not have the Baire property. In turn, $\atree$ and $\satree$ are not Borel.
     \item\label{it:z}  If $\stree\neq \emptyset$ then $\stree$ does not have the Baire property and so $\stree$ is not Borel.
     \item\label{it:k} If $\ktree\neq \emptyset$ then $\ktree$ does not have the Baire property and so $\ktree$ is not Borel.
 \end{enumerate} 

 \end{cor}

\begin{tproof}(\ref{it:nonm}) and (\ref{it:sep}) immediate from Lemma \ref{lm:at1} (1).

(\ref{it:x}) If such an isomorphism class has the Baire property then there is a somewhere co-meager set of trees all isomorphic to a fixed tree with no uncountable branch. This is not possible by Lemma \ref{lm:at1} (2). 

 (\ref{it:y}), (\ref{it:z}) and (\ref{it:k}) again follow from Lemma \ref{lm:at1} (2): these classes are closed under isomorphism classes so must be everywhere non-meager. If they are Baire then they are somewhere comeager and hence contain a tree with an uncountable branch and also a special Aronszajn tree. This leads to a contradiction in case of any of these classes.






\end{tproof}

So, whenever there is a Suslin tree then the set of all Suslin-trees is not Borel (but always $\piii$). Could it be analytic too? We show that this is independent (even assuming CH). 



\begin{prop}
In the Abraham-Shelah model, $\atree\neq \satree\in \diii$ and $\stree\in \diii$ as well.
\end{prop}
\begin{cproof}
Indeed, there are non-special Aronszajn trees (even Suslin-trees) and a tree $T$ is special if and only if it club-embeds no derived subtree of a fixed Suslin tree $S$. Since there are only $\aleph_1$-many such derived subtrees, this gives $\satree\in \piii$ and so $\satree\in \Delta_1^1$ (using Observation \ref{obs:basic}).

In the Abraham-Shelah model, there are only $\aleph_1$ many Suslin trees modulo club isomorphism so fix a representative of each class and collect them as $\mc S$. Now, being Suslin is characterized by being club-isomorphic to some element of $\mc S$ which in turn implies  $\stree\in \siii$ and  $\stree\in \diii$ as well (using again Observation \ref{obs:basic}).
\end{cproof}






\medskip

In the next section, we show that both $\atree$ and $\stree$ are $\piii$-complete if we assume $V=L$.

 \section{Reductions between subsets of $\omg$ and $\aleph_1$-trees}\label{sec:reduce}
 
 


Our first theorem in this section establishes a continuous reduction between stationarity and $\stree$ in a strong form.

 \begin{thm}\label{thm:red1}
  Suppose $\diamondsuit^+$. There is a map $X\mapsto T^X$ from subsets of $\omg$ to the set of downward closed $\aleph_1$-subtrees of $2^{<\omg}$ so that \begin{enumerate}
  \item if $X\cap \alpha=Y\cap \alpha$ then $T^X\uhr \alpha=T^Y\uhr \alpha$,
  \item if $X$ is stationary then    $T^X$ is Suslin, and
  \item if $X$ is non-stationary then $T^X$ has an uncountable branch.
  \end{enumerate}

 \end{thm}

Note that by CH, $2^{<\omg}$ has size $\aleph_1$ so we can easily transform our trees to live on $\omg$. So, we immediately get the following corollary by Theorem \ref{thm:fok}.

  \begin{cor}If $V=L$ then $\atree$ and $\stree$ are both $\piii$-complete.
  \end{cor}
  
  Let us proceed with the proof of the theorem.
 
 
\begin{proof}[Proof of Theorem \ref{thm:red1}]Let is start by fixing a  $\diamondsuit^+$-oracle $\bar N$ i.e., a sequence of elementary submodels $(N_\alpha)_{\alpha<\omg}$ of $(H(\aleph_2),\in,\prec)$ so that $(N_\alpha)_{\alpha<\beta}\in N_\beta$ and $\bar N$ witnesses $\diamondsuit^+$.

 
 
 \medskip
 
 Given $X\subs \omg$, we construct the downward closed subtree $T^X\subs 2^{<\omg}$ level by level in an induction, so that $(T^X_\alpha)_{\alpha< \beta}\in N_{\beta+1}$ for all $\beta<\omg$. In each step, we shall add a new countable level to the tree constructed so far. In successor steps $\beta=\alpha+1<\omg$, we simply take the binary extension $T^X_\beta=\{s\smf i:s\in T^X_\alpha,i<2\}$ of $T^X_\alpha$.
 
 Now, assume $\beta\in \omg$ is a limit ordinal. If $\beta\in \omg\setm X$ then we let $$T^X_\beta=\{b\in 2^\beta\cap N_{\beta}: b\uhr \alpha\in T^X_{\alpha} \textmd{ for all }\alpha<\beta\}.$$ In other words, we continue all branches through $T^X_{<\beta}$ which are in $N_\beta$. Since $N_\beta$ is countable, this is a valid extension and is defined in $N_{\beta+1}$. Moreover, any $s\in T^X_{<\beta}$ has an extension in $T^X_\beta$ (i.e., the tree remains pruned).
 
 Second, if $\beta\in X$ then, working in $N_{\beta+1}$, we make sure that 
 \begin{enumerate}
  \item any $s\in T^X_{<\beta}$ has an extension in $T^X_\beta$, and
  \item for any $A\in N_\beta$ so that $A\subs T^X_{<\beta}$ is a maximal antichain, any  new element $t\in T^X_\beta$ is above a node in $A$.
 \end{enumerate}

  Since $N_\beta$ is countable and $N_\beta,T^X_{<\beta}\in N_{\beta+1}$, the level $T^X_\beta$ can be constructed  in $N_{\beta+1}$ (just as in the classical construction of Suslin trees \cite{kunen}). 
 
 This induction certainly defines an $\aleph_1$-tree $T^X$ for any $X\subset \omg$. The next two claims will conclude the proof of the theorem.
 
 \begin{clm}
If $X$ is stationary then $T^X$ is Suslin.  
 \end{clm}
\begin{proof}
Suppose that $A\subs T^X$ is a maximal antichain. Since $\bar N$ guesses $A$ at club many points, we can find some $\beta\in X$ so that $A\cap T^X_{<\beta}\in \mc N_\beta$ and $A\cap T^X_{<\beta}$ is a maximal antichain in $T^X_{<\beta}$. So, at stage $\beta$, we made sure that any $t\in T^X_\beta$ is above some element of $A\cap T^X_{<\beta}$. In turn, we must have $A\subs T^X_{<\beta}$ and so $A$ is countable.

\end{proof}

 \begin{clm}
  If $X$ is non stationary then $T^X$ has an uncountable branch.
 \end{clm}

\begin{proof}
There is some club $B$ that is disjoint from $X$, and there is a club $C\subs \omg$ so that $C\cap \alpha,B\cap \alpha\in   N_\alpha$ whenever $\alpha\in C$. In particular, $C\cap B\cap \alpha\in N_\alpha$ for any $\alpha \in C\cap B$. Let  $\{\beta_\alpha:\alpha<\omg\}$ be the increasing enumeration of  $C\cap B$.

We construct $t_\alpha\in T^X_{\beta_\alpha}$ so that 
\begin{enumerate}
 \item $t_{\alpha'}<t_\alpha$ for all $\alpha'<\alpha<\omg$, 
 \item $(t_{\alpha'})_{\alpha'<\alpha}\in N_{\beta_\alpha}$, and
 \item $t_\alpha$ is the $\prec$-minimal element of $2^{\beta_\alpha}$ that extends all elements in the chain $(t_{\alpha'})_{\alpha'<\alpha}$.
\end{enumerate}
In limit steps, note that $\bar N\uhr B\cap C\cap \beta_\alpha=(N_{\beta_{\alpha'}})_{\alpha'<\alpha}\in N_{\beta_\alpha}$. So the sequence $(t_{\alpha'})_{\alpha'<\alpha}$ is in $N_{\beta_\alpha}$ too since it can be uniquely defined from $\bar N\uhr B\cap C\cap \beta_\alpha$. As $\beta_\alpha\notin X$, we sealed all branches that are in $N_{\beta_\alpha}$, so $t_\alpha\in T^X_{\beta_\alpha}$ as well. In turn, $(t_{\alpha'})_{\alpha'\leq \alpha}\in N_{\beta_\alpha+1}$.


\end{proof}

This proves the theorem.
\end{proof}

It would be interesting to see whether a single Suslin-tree suffices to construct such a reduction or if weaker reductions (say between stationarity and $\atree$) exist under weaker assumptions than $\diamondsuit^+$.

\medskip

Next, we present a variation that reduces non-stationarity to $\satree$.

  \begin{thm}
  Suppose $\diamondsuit^+$. There is a map $X\mapsto T^X$ from subsets of $\omg$ to the set of downward closed $\aleph_1$-subtrees of $2^{<\omg}$ so that \begin{enumerate}
  \item if $X\cap \alpha=Y\cap \alpha$ then $T^X\uhr \alpha=T^Y\uhr \alpha$,
  \item if $X$ is stationary then    $T^X$ is Suslin, and
  \item if $X$ is non-stationary then $T^X$ is a special Aronszajn tree.

  \end{enumerate}

 \end{thm}
 
 \begin{tproof}
 The idea is very similar: we build $T^X$ level by level and aim for a Suslin tree at stages $\beta\in X$. However, if $\beta\in \omg\setm X$ then we shall try to make $T^X$ special. In fact, we will add the new level $T_\beta^X$ so that any nice enough monotone map $\varphi:T^X_{<\beta}\to \mb Q$ that is also in $N_\beta$ has an extension to $T_\beta^X$. We will need to make sure that any new node at level $\beta$ works simultaneously for all such specializing maps which inspires the definition of a \emph{specializing pair} below. Intuitively, we not just assign a rational number $\varphi(s)$ to a tree node $s$  but also a promise (in the form of a positive number $\delta(s)$) to keep all values $\varphi(t)$ close to $\varphi(s)$ whenever $t$ is above $s$. The details follow below.
 
Fix a  $\diamondsuit^+$-oracle $\bar N$. Given $X$, we construct the downward closed subtree $T^X\subs 2^{<\omg}$ so that $(T^X_\alpha)_{\alpha< \beta}\in N_{\beta+1}$ for all $\beta<\omg$. If $\beta$ is successor or if  $\beta\in X$ then, working in $N_{\beta+1}$, we repeat the construction in the previour theorem. We make sure that 
 \begin{enumerate}
  \item any $s\in T^X_{<\beta}$ has an extension in $T^X_\beta$, and
  \item for any $A\in N_\beta$ so that $A\subs T^X_{<\beta}$ is a maximal antichain, any  new element $t\in T^X_\beta$ is above a node in $A$.
 \end{enumerate}
 
 This will certainly make sure that $T^X$ is Suslin whenever $X$ is stationary. Let us turn to the construction when $\beta$ is a limit ordinal from $\omg\setm X$.
 
 First, a new definition. Given any tree $T$ of height $\beta\leq \omg$, we say that  $(\varphi,\delta)$ is a specializing pair on $T$ if
 \begin{enumerate}[(i)]
     \item $\varphi:T\to \mb Q$ is monotone (in particular, $T$ is special),
     \item if $s<t\in T$ then $|\varphi(s)-\varphi(t)|<\delta(s)$, and
     \item if $\alpha<\beta<\htt(T)$, $s\in T_\alpha$ and $\Delta>0$ then there is some $t\in T_\beta$ above $s$ so that $$|\varphi(s)-\varphi(t)|<\Delta \textmd{ and } \delta(t)<\Delta.$$
 \end{enumerate}
 Our first goal is the following: given a countable tree $T$  of limit height $\beta$ and a countable family $\Phi$ of specializing pairs for $T$, we show that there is a cofinal branch $b$ through $T$ so that any $(\varphi,\delta)\in \Phi$ can be extended to $t^*=\cup b$.
 
 \begin{lemma}\label{lm:newnode} Suppose that $T\subset \oo^{<\omg}$ is a countable tree  of limit height $\beta$ and $\Phi$ is a countable family of specializing pairs for $T$. Fix some $s\in T$, $(\varphi^*,\delta^*)\in \Phi$ and $\Delta>0$. Then there is a cofinal, downward closed branch $b\subset T$ containing $s$ so that 
 $$\sup_{s\leq t\in b}|\varphi^*(s)-\varphi^*(t)|<\Delta$$ and for any $t'\in b^\downarrow$ and $(\varphi,\delta)\in \Phi$,
     $$\sup_{t'\leq t\in b}|\varphi(t')-\varphi(t)|<\delta(s').$$
     
 \end{lemma}
Indeed, if the lemma holds and we set  $t^*=\cup b$ and define $\varphi(t^*)=\sup_{t\in b}\varphi(t)$ for $(\varphi,\delta)\in \Phi$ then $(\varphi,\delta)$ still satisfies (i) and (ii) from the definition of a specializing pair on the extra node $t^*$. Moreover, if we set $\delta^*(t^*)=\Delta/2$ then for this particular $s$ and $\Delta$, we made sure that condition (iii)  for $(\varphi^*,\delta^*)$ is witnessed by $t^*$. By repeating the procedure of Lemma \ref{lm:newnode} for all the elements of $\Phi$ with all possible rational $\Delta>0$, we get the following.

 \begin{lemma}\label{lm:ext2} Suppose that $T\subset \oo^{<\omg}$ is a countable tree  of limit height $\beta$ and $\Phi$ is a countable family of specializing pairs for $T$. Then $T$ has a pruned end-extension $T^*$ of height $\beta+1$ so that any $(\varphi,\delta)\in \Phi$ has an extension $(\bar \varphi,\bar \delta)$ that is a specializing pair on $T^*$.
 \end{lemma}

\begin{cproof}[Proof of Lemma \ref{lm:newnode}] Given $s<t\in T$ and ${\bf p}\in \Phi$, let $\Delta_{\bf p}(s,t)=\delta(s)-|\varphi(s)-\varphi(t)|$. This measures how much slack we have after jumping from $s$ to $t$. List all ${\bf p}\in \Phi$ and $t\in T$ as $({\bf p}_n,t_n)$, each infinitely often. 
We define $s<s_0<s_1<s_2<\dots$ so that $s_n\in T_{\alpha_n}$ for some fixed $(\alpha_n)_{n\in \oo}$ cofinal sequence in $\alpha$. 
First, we pick $s_0$ so that $$\varphi^*(s_0)-\varphi^*(s)<\Delta/2$$ and $$\delta^*(s_0)<\Delta/2.$$ This can be done by (iii). Moreover, note that no matter how we pick $t$ above $s_0$, we will always have $\varphi^*(t)-\varphi^*(s_0)<\Delta/2$ and so for any cofinal branch $b$ above $s_0$, $$\sup_{s\leq t\in b}|\varphi^*(t)-\varphi^*(s)|<\Delta$$ by the triangle inequality.

Given $s_n$, we pick $s_{n+1}$ as follows: look at $({\bf p}_n,t_n)$ and assume $t_n< s_n$ (if the latter fails, pick $s_{n+1}$ arbitrarily). Now, look at $\tilde \Delta=\Delta_{{\bf p}_n}(t_n,s_n)$ and pick $s_{n+1}$
so that $$\varphi_n(s_{n+1})-\varphi_n(s_n)<\tilde \Delta/2$$ and $$\delta_n(s_{n+1})<\tilde \Delta/2$$ where ${\bf p}_n=(\varphi_n,\delta_n)$. This is again possible by (iii). As before, for any $t$ above $s_{n+1}$, we will always have $\varphi_n(t)-\varphi_n(s_{n+1})<\tilde \Delta/2$ and so for any cofinal branch $b$ above $s_{n+1}$, $$\sup_{s_{n+1}\leq t\in b}\varphi_n(t)-\varphi_n(t_n)<\delta_n(t_n)$$ by the triangle inequality and unwrapping the definition of $\tilde \Delta$.

The final branch $b$ is given by the downward closure of $(s_n)_{n<\oo}$. Since any  $t'\in b$ and specializing pair ${\bf p}\in \Phi$ was considered infinitely often during the construction, we clearly satisfied the requirements.

\end{cproof}

Finally, we can describe what happens in the construction of $T^X_\beta$ at limit steps $\beta\in \omg\setm X$. Working in $N_{\beta+1}$, we consider the tree $T^X_{<\beta}$ and $\Phi_\beta=\{(\varphi,\delta)\in N_\beta:(\varphi,\delta) $ is a specializing pair for $T^X_{<\beta}\}$. Now, applying Lemma \ref{lm:ext2}, we add a new level to $T^X_{<\beta}$ so that any specializing pair from $\Phi_\beta$ extends to $T^X_{\leq \beta}$. 

This ends the construction of $T^X$ and we are left to prove the following.

\begin{clm}
If $X$ is non-stationary then $T^X$ is a special Aronszajn tree.
\end{clm}
\begin{cproof}
In fact, we prove that $T^X$ has a specializing pair. By our assumption on $X$, we can find a club $C\subset \omg\setm X$ so that for any $\beta\in C$, $X\cap \beta,C\cap \beta\in N_\beta$. Let $\{\beta_\alpha:\alpha<\omg\}$ be the increasing enumeration of $C$ and we define an $(\varphi_\alpha,\delta_\alpha)$ for $\alpha<\omg$ so that
\begin{enumerate}
    \item $(\varphi_\alpha,\delta_\alpha)\in N_{\beta_\alpha+1}$ is a specializing pair on $T^X_{<\beta_\alpha}$ uniquely definable from $X\cap \beta_\alpha,C\cap \beta_\alpha,\bar N\uhr \beta_\alpha$,
    \item for $\alpha<\alpha'<\omg$, $(\varphi_{\alpha'},\delta_{\alpha'})$ extends $(\varphi_\alpha,\delta_\alpha)$.
\end{enumerate}
As before, $N_{\beta_{\alpha'}}$ has all the information to reconstruct the sequence $((\varphi_\alpha,\delta_\alpha))_{\alpha<\alpha'}$ and so $$(\bigcup_{\alpha<\alpha'}\varphi_\alpha,\bigcup_{\alpha<\alpha'}\delta_\alpha)\in \Phi_{\beta_{\alpha'}}.$$ Since $\beta_{\alpha'}\notin X$, we made sure that this specializing pair has an extension to level $T_{\beta_{\alpha'}}^X$ which gives $(\varphi_{\alpha'},\delta_{\alpha'})$.
\end{cproof}

 \end{tproof}

  \begin{cor}If $V=L$ then $\satree$ is $\siii$-complete.
  \end{cor}
  
\section{Kurepa trees}\label{sec:kurepa}

Our goal in this section is to show the following.

\begin{thm}$(V=L)$ The set $\ktree$ of all Kurepa trees is $\Pi^1_2$-complete.
\end{thm}

We prove the above result through a series of lemmas. First, we will build on the following representation of $\Pi^1_2$ sets.

\begin{lemma}
If $A$ is a $\Pi^1_2$ subset of $\omg^\omg$ then for some 
$\Sigma_1$ formula $\phi$ and some parameter $P\in \omg^\omg$, the following are equivalent
\begin{enumerate}
    \item $X\in A$,
    \item $\sup \{i<\oo_2 : 
L_{\oo_2}[X]\models \phi(X,P,i)\}=\omega_2$.
\end{enumerate}
\end{lemma}
\begin{cproof}
Since $A$ is $\Pi^1_2$, we can find a Borel set $B$ so that  $X$ is in $A$ if and only if $$\forall Y\; \exists Z\; (X,Y,Z) \in B.$$ 
 Let $f:\omega_2 \to \omg^\omg$ be 
a bijection which is $\Sigma_1$ over $L_{\omega_2}$ and choose a 
$\Sigma_1$ formula $\psi$ with parameter $P$ in $\omg^\omg$ so that $(X,Y,Z)\in B$ if and only if  $L_{\oo_2}\models \psi(X,Y,Z,P)$. 

Then
$X\in A$ if and only if 
\begin{equation} \label{eq1}
    L_{\omega_2}\models \forall j\; \exists k\; \psi(X,f(j),f(k),P).
\end{equation}

Let $\phi(X,P,i)$ be the formula 
$$X,P\in L_i \wedge L_i\models \forall j\; \exists k\; \psi(X,f(j),f(k),P).$$ Then (\ref{eq1}) (and so $X\in A$ as well) is equivalent to

$$\sup \{i < \omega_2 : L_{\oo_2}\models \phi(X,P,i)\}=\omega_2,$$ as desired. 
\end{cproof}


Fix some $X\subset \omg^\omg$ with corresponding $\Sigma_1$ formula $\phi$ and parameter $P$.  First, note that $$\{i<\oo_2 : 
L_{\omega_2}\models \phi(X,P,i)\}=\{i<\oo_2 : 
\exists \beta<\oo_2 (L_{\beta}\models \phi(X,P,i))\}.$$
Our plan is to form an $\aleph_1$-tree $T=T_X$ consisting of triples $(\bar \alpha,\bar i,\bar \beta)$ from $\omg$ which resemble a  triple $(\omg,i,\beta)$ from $\oo_2$ with $L_{\beta}\models \phi(X,P,i)$. Distinct uncountable branches in the tree $T$ will correspond to distinct triples $(\omg,i,\beta)$. In turn,  whether $T$ has $\omega_2$ many branches (i.e., if $T$ is Kurepa) will characterize whether $X\in A$. This will prove that $\ktree$ is $\Pi^1_2$-complete.

\medskip

We will say that a triple $(\alpha,i,\beta)$ from $\oo_2$ is \textit{good} (with respect to $X,\phi$ and $P$) if  

\begin{enumerate}
    \item $\alpha<i<\beta<\oo_2$,
    \item\label{it:size} $L_\beta\models \alpha=\omg$,
    \item $\beta$ is the least limit ordinal so that 
    \begin{enumerate}
        \item $X \cap \alpha, P \cap 
\alpha\in L_\beta$,
\item $L_\beta\models |i|=\alpha$,
\item $L_\beta\models \phi(X \cap \alpha,P \cap \alpha,i)$.
    \end{enumerate}
\end{enumerate}

Note that $\alpha\leq \omg$; in case of equality,  $X \cap \alpha=X$ and $P \cap \alpha=P$. If $\alpha<\omg$ then $\beta<\omg$ as well by (\ref{it:size}).
\medskip

The next claim should be clear from the minimality of $\beta$.

\begin{clm}
If $(\alpha,i,\beta)$ is good then the Skolem hull of $\alpha \cup \{\alpha,X \cap \alpha,P\cap \alpha,i\}$ in $L_\beta$ 
is all of $L_\beta$.
\end{clm}

\medskip
\newcommand{\tri}{\triangleleft}

Next, we define an ordering on good triples: we write $$(\bar \alpha,\bar i,\bar \beta)\tri (\alpha,i,\beta)$$ if $\bar \alpha<\alpha$ and there is a (unique) elementary embedding $$\psi: L_{\bar \beta}\hookrightarrow L_\beta$$ so that 
\begin{enumerate}
    \item $\psi\uhr \bar \alpha$ is the identity,
    \item $\psi(\bar \alpha)=\alpha$, 
    \item $\psi(X\cap \bar \alpha)=X\cap \alpha$, $\psi(X\cap \bar \alpha)=P\cap \alpha$ and
    \item $\psi(\bar i)=i.$
\end{enumerate}

Note that $L_{\bar \beta}=\psi^{-1}(L_\beta)$ is the transitive closure of $H=\textmd{Hull}(\bar \alpha \cup \{\alpha, X\cap \alpha, P\cap\alpha\})$ in $L_\beta$. In turn, for a given good triple $(\alpha,i,\beta)$ and $\bar \alpha<\alpha$, there is $\bar i,\bar \beta$ so that   $(\bar \alpha,\bar i,\bar \beta)<(\alpha,i,\beta)$ iff for the above hull $H$, $H\cap \alpha = \bar \alpha$.

Let us summarize the basic properties of the relation $\tri$.

\begin{clm}\label{clm:basictri}\begin{enumerate}
    \item The relation $\tri$ is transitive.
    \end{enumerate} Moreover, for any good triple $(\alpha,i,\beta)$,
    \begin{enumerate}
    \setcounter{enumi}{1}
    \item  for any $\bar \alpha<\alpha$, there is at most one choice of $\bar i,\bar \beta$ so that $(\bar \alpha,\bar i,\bar \beta)\tri (\alpha,i,\beta)$;
\item the set $$\{\bar \alpha<\alpha:\exists \bar i,\bar \beta\; (\bar \alpha,\bar i,\bar \beta)\tri (\alpha,i,\beta) \}$$ is closed in $\alpha$.
\end{enumerate}

\end{clm}

\begin{clm}
The relation $\tri$ is a tree order on good triples. 

\end{clm}

\begin{tproof}
Suppose that we are given good triples $(\alpha_k,i_k,\beta_k)$ for $k<3$ so that $$(\alpha_0,i_0,\beta_0),(\alpha_1,i_1,\beta_1)\tri(\alpha_2,i_2,\beta_2).$$

If $\alpha_0=\alpha_1$ then $(\alpha_0,i_0,\beta_0)=(\alpha_1,i_1,\beta_1)$; indeed, this follows from the Claim \ref{clm:basictri}.
So we can assume $\alpha_0<\alpha_1$. Now, if  $\psi_j:L_{\beta_j}\hookrightarrow L_{\beta_2}$ witnesses that $(\alpha_j,i_j,\beta_j)\tri(\alpha_2,i_2,\beta_2)$ for $j=0,1$ then $\psi=\psi_1^{-1}\circ \psi_0$ witnesses $(\alpha_0,i_0,\beta_0)\tri (\alpha_1,i_1,\beta_1)$. 

\end{tproof}

For some technical reasons, instead of taking the tree of good triples, we will look at functions associated to good triples and the tree formed by them. For each good triple $(\alpha,i,\beta)$, define a function $f_{(\alpha,i,\beta)}$ with domain $\alpha+1$ as follows:

\[
f_{(\alpha,i,\beta)}(\bar \alpha)=
\begin{cases}
(\alpha,i,\beta) & \textmd{ if } \bar \alpha=\alpha,\\
(\bar \alpha,\bar i,\bar \beta) & \textmd{ if } (\bar \alpha,\bar i,\bar \beta)\tri (\alpha,i,\beta)  \textmd{ for some } \bar i,\bar \beta, \textmd{ and}\\
0 & \textmd{ otherwise.}

\end{cases}
\]
This is well-defined by Claim \ref{clm:basictri}.

\begin{clm}\label{clm:restr}
For any $(\bar \alpha,\bar i,\bar \beta)\tri (\alpha,i,\beta)$, $f_{(\alpha,i,\beta)}\uhr \bar \alpha+1=f_{(\bar \alpha,\bar i,\bar \beta)}$.
\end{clm}
\begin{cproof}
This follows immediately from the fact that $\tri$ is a tree order.
\end{cproof}

We let $T_X$ be the set of function $f_{(\alpha,i,\beta)}\uhr \bar \alpha+1$ where $(\alpha,i,\beta)$ is a good triple (with respect to $X,\phi,P$) of countable ordinals and $\bar \alpha \leq \alpha$.

\begin{clm}
$T_X$ is a tree of height at most $\omg$ and countable levels.
\end{clm}
\begin{cproof}We prove that every level of $T_X$ is countable by induction.
Elements of $T_X$ at level $\bar \alpha$ are of the form $f=f_{(\alpha,i,\beta)}\uhr \bar \alpha+1$.
These functions either satisfy (i) $f(\bar \alpha)=0$ or (ii) $f(\bar \alpha)=(\bar \alpha,\bar i, \bar \beta)$.
In case (i), $f$ must be constant 0 on an end-segment of $\bar \alpha$ and so $f$ is completely determined by the previous levels so by induction, there are only countably many choices for $f$. In case (ii), we note that $f=f_{(\bar \alpha,\bar i, \bar \beta)}$ by Claim \ref{clm:restr}. Finally, note that there are only countably many possibilites for the value of $\bar \beta$ since for any large enough $\gamma<\omg$, $L_\gamma\models |\alpha|\leq \aleph_0.$ This proves that level $\bar \alpha$ of $T_X$ must be countable.

\end{cproof}

The next lemma will conlcude the proof of the theorem.

\begin{clm}
$T_X$ has $\aleph_2$ uncountable branches if and only if the set  $$\sup \{i<\oo_2 : 
L_{\omega_2}[X]\models \phi(X,P,i)\}=\oo_2.$$
\end{clm}
\begin{tproof}
First, assume that $L_{\omega_2}\models \phi(X,P, i)$ for some $i>\omg$. Pick the minimal $\beta$ so that  $L_{\beta}\models \phi(X,P,i)$ and so $(\omg,i,\beta)$ is a good triple with respect to $X,\phi,P$.

\begin{subclaim}
$B=\{f_{(\omg,i,\beta)}\uhr \alpha+1:\alpha<\omg\}$ is an uncountable branch in $T_X$.
\end{subclaim}

Moreover, the branch $B$ uniquely determines the triple $(\omg,i,\beta)$ by the following claim.

\begin{subclaim}
$(L_\beta,X,P,i)$ is the direct limit of $(L_{\bar \beta},X\cap \alpha, P\cap \alpha,\bar i)$ for $(\bar \alpha,\bar i,\bar \beta)\tri (\omg,i,\beta)$.
\end{subclaim}

 Thus distinct good triples $(\omg,i,\beta)$ correspond to different branches in $T_X$.
\medskip

Conversely, suppose that we have a branch $B$ in $T_X$, which is not eventually 0. Now, the
direct limit of the 
$(L_{\bar \beta},X \cap \bar \alpha,P \cap \bar \alpha,\bar i)$ for $(\bar \alpha,\bar i,\bar \beta)\in \ran \cup B$  yields some $(L_\beta,X,P,i)$  and a good triple $(\omega_1,i,\beta)$. Moreover, $B$ is the restriction of $f_{(\omega_1,i,\beta)}$. 
Distinct $\omega_1$-branches yield distinct triples $(\omega_1,i,\beta)$
and since $\beta$ is uniquely determined by $i$, we get 
$\omega_2$-many $i$ such that $\phi(X,P,i)$ holds in $L_{\omega_2}$.


\end{tproof}

\section{Open problems and future goals}\label{sec:questions}

A positive answer to the following question would show that there are no ZFC reductions between stationarity and $\atree$.

\begin{quest}
Is it consistent with CH that all Aronszajn trees are special and there are no Canary trees?
\end{quest}

Regarding Kurepa-trees, the following remains open.

\begin{quest}
Can 
$\ktree$ be $\Delta_2^1$ and nonempty?
\end{quest}

\begin{quest}
What is the complexity of $\ktree$ under $MA_{\aleph_1}$ (given such trees exist)?
\end{quest}

The following would also be very interesting.

\begin{quest}
Find a \textit{natural} class of structures $\mc X$ in $\siii\setm \diii$ or $\piii\setm \diii$ which is \emph{not} complete for its complexity class.
\end{quest}

The reason we ask for a natural class is that under $V=L$, one can build such \textit{artificial} examples (and for inaccessible cardinals there are even natural examples) but we wonder if there are more combinatorial examples on $\omg$. Also, Harrington proved that consistently no such intermediate classes exist.

\medskip
Yet another axiom to consider in more detail is $PFA(S)$ for coherent Suslin-trees $S$ (see e.g., \cite{tall2017pfa}). Such models allow the existence of Suslin-trees while share many properties with models of the proper forcing axiom.

\begin{quest}
How does $PFA(S)$ affect the complexity of the classes $\atree, \satree$ and $\stree$?
\end{quest}

Once we force with the Suslin tree $S$ over a model of $PFA(S)$, the resulting extension has no Suslin trees any more and in fact, any two Aronszajn trees will be club-isomorphic \cite{yorioka2017club}. In turn, the complexity of $\atree$ and $\satree$ is $\Delta^1_1$ as mentioned in Figure \ref{fig:table}.

It would certainly be interesting to see to what extent our results generalise to higher cardinals above $\omg$. In particular, we mention a recent result of  Krueger \cite{krueger2018club} on the club-isomorphism of higher Aronszajn trees that could substitute the Abraham-Shelah model. The construction schemes developed by Brodsky, Lambie-Hanson and Rinot \cite{brodsky2017microscopic, rinot2017higher, lambie2017aronszajn} for higher Suslin and Aronszajn trees also seems rather relevant. The theorem of Jensen on CH and all Aronszajn trees being special was recently generalized to higher cardinals in a breakthrough result by Asper\'o and Golshani \cite{aspero2018special}. Definability of $\textmd{NS}_\kappa$ on successor cardinals was investigated by Friedman, Wu and Zdomskyy \cite{friedman2015delta}.

\medskip

\medskip

We believe that a similar analysis of  other classes of structures on $\omega_1$ is well worth exploring. To name the most natural candidates, we would be interested in the following: 
\begin{enumerate}
    \item \emph{Graphs and more generally colourings} $c:[\omg]^2\to r$ with $r\leq \omg$. One might look at graphs with chromatic or colouring number $\omg$; or strong colourings that witness the failure of square bracket relations (i.e., $\omg \not\to [\omg]^2_\omg$). Hypergraphs and various set-systems are also natural candidates.
    
    \item \textit{Ladder systems on $\omg$}. Natural classes are ladder systems with various guessing properties (i.e., $\clubsuit$ sequences or club guessing sequences); and ladder systems with or without the uniformization property.
    
    \item \textit{Linear orders on $\omg$}. For Aronszajn and Suslin-lines, our analysis most likely yields the appropriate  complexities but we did not address the important class of \emph{Countryman lines}.
    
    \item \textit{Forcing notions}. We can consider various classes of forcing posets on $\omg$ such as ccc, Knaster, $\sigma$-centered, $\sigma$-linked or proper partial orders. 

\end{enumerate}

Finally, it will be very natural to consider the classical \textit{equivalence relations} on these classes and to find Borel definable reductions between them. To mention a few, we name the order isomorphism of trees and linear orders;  the notion of club-isomorphism between trees;  graph isomorphism; bi-embeddability of various structures; or forcing equivalence of posets.

%

\bibliographystyle{plain}
\bibliography{thesis}

\end{document}